\documentclass[12pt]{amsart}

\usepackage{graphicx}
\usepackage{fullpage}
\usepackage{etoolbox}
\usepackage[inline]{enumitem}
\usepackage[utf8]{inputenc}
\usepackage{amsmath}  
\usepackage{amssymb}     
\usepackage{amsthm}
\usepackage{bm}
\usepackage{mathtools}
\usepackage{fixmath}
\usepackage{fullpage}
\usepackage{tikz}
\usepackage{tikz-3dplot}
\usepackage[colorlinks,linkcolor=blue,citecolor=magenta, pagebackref]{hyperref}
\usepackage[normalem]{ulem}

\usepackage{microtype}

\usetikzlibrary{shapes}

\newtheorem{theorem}{Theorem}[section]
\newtheorem*{theorem-nonumber}{Theorem}
\newtheorem{proposition}[theorem]{Proposition}

\newtheorem{lemma}[theorem]{Lemma}

\theoremstyle{remark}

\newtheorem{example}{Example}[section]

\def\G{\mathcal{G}}
\def\R{\mathbb{R}}
\def\Rext{\overline{\mathbb{R}}}
\def\Z{\mathbb{Z}}
\def\conv{\operatorname{conv}}
\def\cconv{\overline{\operatorname{conv}}}

\def\x{{x}}
\def\convc{\overline{\operatorname{conv}}}

\renewcommand{\geq}{\geqslant}
\renewcommand{\leq}{\leqslant}

\newcommand{\exampleqed}{\hfill $\diamond$}

\title{Geoffrion's theorem beyond finiteness and rationality}

\author{Santanu S. Dey,  Fr\'ed\'eric Meunier, Diego A. Mor\'an R.}

\begin{document}

\begin{abstract}
   Geoffrion's theorem is a fundamental result from mathematical programming assessing the quality of Lagrangian relaxation, a standard technique to get bounds for integer programs. An often implicit condition is that the set of feasible solutions is finite or described by rational linear constraints. However, we show through concrete examples that the conclusion of Geoffrion's theorem does not necessarily hold when this condition is dropped. We then provide sufficient conditions ensuring the validity of the result even when the feasible set is not finite and cannot be described using finitely-many linear constraints.
\end{abstract}

\keywords{Lagrangian relaxation; Geoffrion's theorem; Non-linear integer programming}

\maketitle


\section{Introduction}

\subsection{Lagrangian relaxation and Geoffrion's theorem} The exact resolution of integer programs relies in a crucial way on bounds that help assess the quality of feasible solutions. A popular systematic and efficient way to produce lower bounds on integer programs formulated as minimization problems, and more generally dual bounds, is  {\em Lagrangian relaxation}, introduced in 1970 by Held and Karp for the resolution of the traveling salesman problem~\cite{held1970traveling,held1971traveling}, which works as follows.

Consider an integer program
\begin{equation}\label{eq:master}
    \begin{array}{rl}
        \text{minimize} & c \cdot x \\
        \text{s.t.} & Ax \geq b \\
        & x \in X \, ,
    \end{array}
\end{equation}
where $A\in \R^{m\times n}$, $b \in \R^m$, $c \in \R^n$, and $X\subseteq\Z^n$. The function
\[
\G \colon \lambda \in \R_+^m \longmapsto \inf \{ c \cdot x + \lambda \cdot (b-Ax) \colon x \in X \} \in \Rext
\]
is a {\em dual function} of~\eqref{eq:master}. As is well known, $\G(\lambda)$ provides a lower bound on the optimal value of~\eqref{eq:master}. {\em Lagrangian relaxation} involves using such lower bounds to solve~\eqref{eq:master}. 
The method often aims to identify the tightest possible bound achievable in this manner, namely the quantity
\[
v^L \coloneqq \sup \{\G(\lambda) \colon \lambda \in \R_+^m \} \, .
\]

One of the results that best explains the strength of Lagrangian relaxation is Geoffrion's theorem~\cite[Theorem~1]{geoffrion1974lagrangian}. It considers the following optimization problem:
\begin{equation}\label{eq:geoffrion}
    \begin{array}{rl}
        \text{minimize} & c \cdot x \\
        \text{s.t.} & Ax \geq b \\
        & x \in \conv(X) \, ,
    \end{array}
\end{equation}
in the space of the original variables. Denote by $v^\star$ the optimal value of~\eqref{eq:geoffrion} which is defined as the infimum of $c\cdot x$ over the feasible region. Geoffrion's theorem considers the case where $X$ is finite or the set of integer points contained in a rational polyhedron, and states that if~\eqref{eq:master} is feasible, then $v^L=v^\star$. (We remind the reader that a {\em rational polyhedron} is the intersection of finitely many half-spaces, each of which  being described using rational data.) 


The proof of Geoffrion's theorem relies crucially on $\conv(X)$ being polyhedral, which is the case when $X$ is finite or when $X$ is the set of integer points within a rational polyhedron (a consequence of Meyer's theorem~\cite{meyer1974existence}). The theorem does not say anything about the case when $X$ is the set of integer points of an unbounded non-rational polyhedron. The only place in the literature we have been able to find where this fact is explicitly mentioned is a sentence in a paper by Guignard~\cite[p.157]{guignard1998efficient}:
\begin{quote}
``\dots this result may not be true if the constraint matrices are not rational, or more precisely for non-rational polyhedra that are not equal to the convex hull of their extreme points.''
\end{quote}

Our contributions are twofold:
\begin{itemize}
    \item Concrete examples where the conclusion of Geoffrion's theorem does not hold because $X$ is precisely the collection of integer points contained in a polyhedron whose defining inequalities have irrational coefficients.
    \item Sufficient conditions ensuring that the conclusion of Geoffrion's theorem holds, including cases where $X$ is an infinite set of integer points that cannot be described as the integer points in a polyhedron, i.e., $X$ may be the set of integer points in a general convex set.
\end{itemize}

When $X$ is finite or  can be described as the integer points  in a rational polyhedron, its convex hull is closed (also a consequence of Meyer's theorem). This is not necessarily the case when $X$ is the set of integer points in an unbounded non-rational polyhedron. Thus, to investigate Geoffrion's theorem ``beyond finiteness and rationality,'' considering the following optimization problem could also be relevant, which will be confirmed by Example~\ref{ex:ex1} in the next subsection:
\begin{equation}\label{eq:geoffrion-closed}
    \begin{array}{rl}
        \text{minimize} & c \cdot x \\
        \text{s.t.} & Ax \geq b \\
        & x \in \cconv(X) \, ,
    \end{array}
\end{equation}
where $\cconv(X)$ is the closure of $\conv(X)$. Denote by $\bar v^\star$ the optimal value of \eqref{eq:geoffrion-closed}  which is again defined to be the infimum of $c\cdot x$ over the feasible region. For sake of completeness, we state Geoffrion's theorem for the optimization problems~\eqref{eq:geoffrion} and~\eqref{eq:geoffrion-closed} simultaneously.

\begin{theorem-nonumber}[Geoffrion's theorem~\cite{geoffrion1974lagrangian}]
    Suppose that $X$ is finite or formed by the integer points of a rational polyhedron. If~\eqref{eq:master} is feasible, then $v^L = \bar v^\star = v^\star$.
\end{theorem-nonumber}

\subsection{Two examples showing the failure of Geoffrion's theorem beyond finiteness and rationality}\label{subsec:ex}

The first example shows that the vanilla version of Geoffrion's theorem---about the optimization problem~\eqref{eq:geoffrion}---does not hold in full generality when $X$ is the set of integer points of a non-rational polyhedron.

\begin{example}\label{ex:ex1}
    Consider the integer program
    \begin{equation}\label{eq:ex1}
   \begin{array}{rl}
        \text{minimize} & -x \\
        \text{s.t.} & -\sqrt{2}x + y \geq 0 \\
        & -\sqrt{2}x + y \leq 0 \\
        & x, y \in \Z_+ \, .
    \end{array}
    \end{equation}
    
    Note that it is a feasible program since $x=y=0$ is a feasible solution. Set $X = \{(x,y) \in \Z_+^2 \colon -\sqrt{2}x + y \leq 0\}$. See Figure~\ref{fig:ex1} for an illustration.

\begin{figure}[htb]
    \centering
   
\begin{tikzpicture}[scale=0.2]

  \definecolor{halfspaceA}{RGB}{173,216,230} 
  \definecolor{halfspaceB}{RGB}{255,182,193} 

  \draw[->, very thick] (0,0) -- (21,0) node[right] {$x$};
  \draw[->, very thick] (0,0) -- (0,31) node[above] {$y$};

  \foreach \x in {0,5,10,15,20} {
    \draw[thick] (\x,0.4) -- (\x,-0.4) node[below] {\footnotesize \x}; 
  }
  \foreach \y in {5,10,...,30} {
    \draw[thick] (0.4,\y) -- (-0.4,\y) node[left] {\footnotesize \y}; 
  }

  \pgfmathsetmacro{\srtwo}{sqrt(2)}
  \pgfmathsetmacro{\ymax}{30}
  \pgfmathsetmacro{\xmax}{\ymax/\srtwo}

  \fill[halfspaceA,fill opacity=0.5] (0,30) -- (20,30) -- (20,\srtwo*20) -- (0,0) -- cycle;

  \fill[halfspaceB,fill opacity=0.5] (0,0) -- (20,0) -- (20,\srtwo*20) -- (0,0) -- cycle;

  \draw[very thick, black] plot[domain=0:15, samples=100] (\x, {\srtwo*\x});

  \draw[very thick, black, dotted] plot[domain=15:20, samples=100] (\x, {\srtwo*\x});

  \node[black] at (22, {\srtwo*17 + 1}) {\footnotesize $y = \sqrt{2}x$};

  \node[text=blue!80!black] at (6, {\srtwo*5 + 10}) {\footnotesize $-\sqrt{2}x + y \geq 0$};
  \node[text=red!80!black] at (7, {\srtwo*5 - 5}) {\footnotesize $-\sqrt{2}x + y \leq 0$};

  \foreach \x in {1,...,20} {
    \foreach \y in {1,...,30} {
      \fill[gray] (\x,\y) circle (1pt);
    }
  }

\end{tikzpicture}
    \caption{The region defined by the inequality $-\sqrt{2}x+y\geq0$ (resp.\ $-\sqrt{2}x+y\leq0$) is shown in blue (resp.\ red). The ray defined by the intersection of the red and blue regions is shown in black. The set $X$ contains the integral points in the red region (including the points in the $x$-axis). Then the only integer point in $X$ that satisfies the inequality $-\sqrt{2}x+y\geq0$ is (0,0).}
    \label{fig:ex1}
\end{figure}
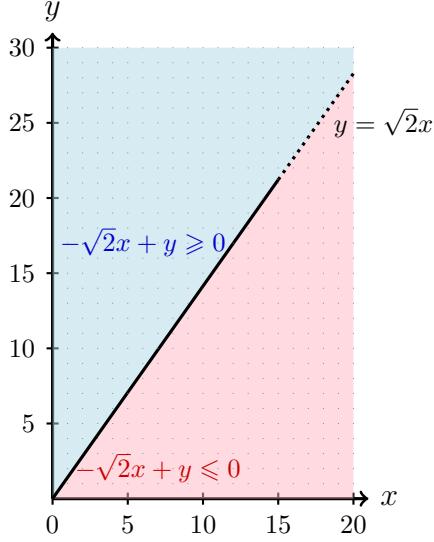


    The dual function in this case is $\G(\lambda) = \inf\{-x + \lambda(\sqrt{2}x-y)\colon (x,y) \in X \}$. For every $x\in \Z_+$, there exists $y \in \Z_+$ such that $(x,y) \in X$ and $\sqrt{2}x - y \leq 1$. Thus, $\G(\lambda) \leq \inf\{-x+\lambda \colon x \in \Z_+\}$, i.e., $\G(\lambda) = -\infty$, for every $\lambda \in \R_+$. In other words, $v^L = -\infty$.

    On the other hand, we have here $v^\star = \inf\{-x \colon -\sqrt{2}x + y \geq 0, \, (x,y) \in \conv(X) \}$. The equality
    \[
    \conv(X) = \{(x,y)\in\R_+^2 \colon -\sqrt{2}x + y \leq 0 \} \setminus \{(x,y)\in\R_+^2 \colon y = \sqrt{2}x, \, x > 0 \}
    \]
    holds (for instance, see the set $K^1$ in Example~1 in~\cite{dey2013some}).
    Hence, $v^\star = \inf\{-x \colon x = y = 0\} = 0$.

    The program~\eqref{eq:ex1} is therefore such that $v^L < v^\star$. \exampleqed
\end{example}

In Example~\ref{ex:ex1}, the closure of $\conv(X)$ is actually $\{(x,y)\in\R_+^2 \colon -\sqrt{2}x + y \leq 0 \}$. Taking in~\eqref{eq:geoffrion} the closure of $\conv(X)$ instead would ensure the validity of Geoffrion's theorem for this example, i.e., $v^L = \bar v^\star$. However, even by considering~\eqref{eq:geoffrion-closed}, there are still cases where Geoffrion's theorem does not hold when $X$ is the set of integer points of an unbounded non-rational polyhedron.


\begin{example}\label{ex:ex2}
    Let $P \coloneqq \conv (P^1 \cup P^2)$ where 
    \[
    \begin{array}{rcl}
    P^1 & \coloneqq & \bigl\{ (x,y,z) \in \R^3 \colon \sqrt{2}x - y \geq 0, \, y\geq 0, \, z = 0\bigr\} \\[1ex]
    P^2 & \coloneqq & \bigl\{(x,y,z) \in \R^3 \colon \sqrt{2}x - y \geq 0, \, y \geq 0, \, x \geq 1/2, \,  z = 1 \bigr\} \, .
    \end{array}
    \]
    The recession cones of $P^1$ and $P^2$ being equal 
    implies $P$ is a polyhedron~\cite{conforti2014integer}. Consider the integer program
    \begin{equation}\label{eq:ex2}
   \begin{array}{rl}
        \text{minimize} & -z \\
        \text{s.t.} & -\sqrt{2}x + y \geq 0 \\
        & (x,y,z) \in X \, ,
    \end{array}
    \end{equation}
    where $X \coloneqq P \cap \Z^3$. See Figure~\ref{fig:ex2} for an illustration.



\tdplotsetmaincoords{75}{205}
\begin{figure}
    \centering
 
\begin{tikzpicture}[tdplot_main_coords, scale=1]

  \draw[->, thick] (0,0,0) -- (5.8,0,0) node[anchor=north east]{$x$};
  \foreach \x in {1,2,3,4,5}
    \draw[thick] (\x,0,0) -- (\x,0,0.1) node[above]{\x};

  \draw[->, thick] (0,0,0) -- (0,8.5,0) node[anchor=north west]{$y$};
  \foreach \y in {1,2,3,4,5,6,7,8}
    \draw[thick] (0,\y,0) -- (-0.15,\y,0) node[right]{\y};

  \draw[->, thick] (0,0,0) -- (0,0,2) node[anchor=south]{$z$};
  \foreach \z in {1}
    \draw[thick] (0,0,\z) -- (-0.15,0,\z) node[right]{\z};

  \fill[blue!30,opacity=1]
    (0,0,0)
    -- (5,0,0)
    -- (5,{sqrt(2)*5},0)
    -- (0,0,0)
    -- cycle;

  \fill[red!30,opacity=0.8]
    (0.5,0,1)
    -- (5,0,1)
    -- (5,{sqrt(2)*5},1)
    -- (0.5,{sqrt(2)*0.5},1)
    -- cycle;

  \fill[orange!40,opacity=0.4]
    (0,0,0) -- (5,0,0) -- (5,0,1) -- (0.5,0,1) -- cycle;

  \fill[orange!40,opacity=0.4]
    (5,0,0) -- (5,{sqrt(2)*5},0) -- (5,{sqrt(2)*5},1) -- (5,0,1) -- cycle;

  \fill[orange!40,opacity=0.4]
    (5,{sqrt(2)*5},0) -- (0,0,0) -- (0.5,{sqrt(2)*0.5},1) -- (5,{sqrt(2)*5},1) -- cycle;

  \fill[orange!40,opacity=0.4]
    (0,0,0) -- (0.5,0,1) -- (0.5,{sqrt(2)*0.5},1) -- cycle;


\draw[thick] (0,0,0) -- (5,{sqrt(2)*5},0);
\draw[thick] (0,0,0) -- (5,0,0);
\draw[thick,dashed] (5,0,0) -- (5,{sqrt(2)*5},0);
\draw[thick] (0.5,{sqrt(2)*0.5},1) -- (5,{sqrt(2)*5},1);
\draw[thick] (0.5,{sqrt(2)*0.5},1) -- (0.5,0,1);
\draw[thick] (0.5,0,1) -- (5,0,1);
\draw[thick,dashed] (5,0,1) -- (5,{sqrt(2)*5},1);

  \draw[thick] (0,0,0) -- (0.5,0,1);
  \draw[thick,dashed] (5,0,0) -- (5,0,1);
  \draw[thick,dashed] (5,{sqrt(2)*5},0) -- (5,{sqrt(2)*5},1);
  \draw[thick] (0,0,0) -- (0.5,{sqrt(2)*0.5},1);

  \node[red] at (4.5,1,1) {\large ${P^2}$};
  \node[blue] at (4.5,5,0) {\large ${P^1}$};
   \node[black] at (4,7.5,0) {\tiny $\{(x,y,z)\in\mathbb{R}^3\colon y=\sqrt{2}x,\ z=0\}$};

\foreach \z in {0}
  \foreach \x in {0,...,5}
    \foreach \y in {0,...,8}
      {
        \pgfmathsetmacro{\limit}{sqrt(2)*\x}
        \pgfmathparse{\y <= \limit ? 1 : 0}
        \ifnum\pgfmathresult=1
          \filldraw[fill=blue!70, draw=blue, line width=0.3pt](\x,\y,\z) circle (0.03);
        \fi
      }
\foreach \z in {1}
  \foreach \x in {0,...,5}
    \foreach \y in {0,...,8}
      {
        \pgfmathsetmacro{\limit}{sqrt(2)*\x}
        \pgfmathparse{\y < \limit ? 1 : 0}
        \ifnum\pgfmathresult=1
          \filldraw[fill=red!70, draw=red, line width=0.3pt] (\x,\y,\z) circle (0.03);
        \fi
      }
        
\end{tikzpicture}
   \caption{The set of points considered in Example~\ref{ex:ex2}.}
    \label{fig:ex2}
\end{figure}
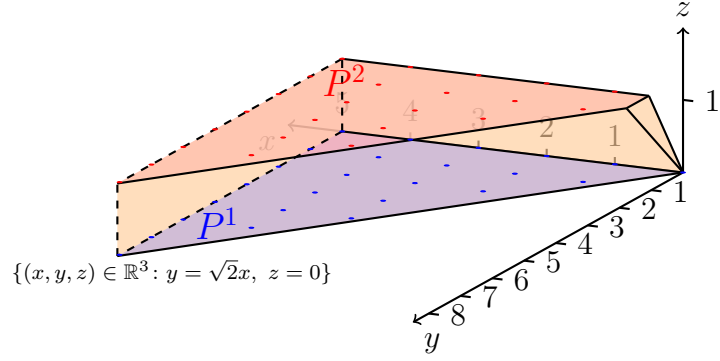


    The dual function in this case is $\G(\lambda) = \inf\{-z + \lambda(\sqrt{2}x - y)\colon (x,y,z) \in X \}$. Notice that $\G(\lambda) \geq -1$ for every $\lambda \in \R_+$. Also, there exists a sequence $(x_i,y_i, 1)$ of points of $X$ such that $\lim_{i \to +\infty}\sqrt{2}x_i - y_i = 0$, which shows that $\G(\lambda) \leq -1$ for every $\lambda \in \R_+$. In other words, $v^L = -1$.

    On the other hand, we have here $\bar v^\star = \inf\{ - z \colon -\sqrt{2}x + y \geq 0, \, (x,y,z) \in \cconv(X) \}$. We will prove now that 
    \begin{equation}\label{eq:inters}
    \cconv(X) \cap \{(x,y,z) \colon -\sqrt{2}x + y \geq 0\} = \{(x,y,z) \colon \sqrt{2}x - y = 0, \, z = 0\}\, . 
    \end{equation} From this equality, it is immediate to check that $\bar v^\star = 0$, and therefore that the instance~\eqref{eq:ex2} is such that $v^L < \bar v^\star$.

    Let us prove~\eqref{eq:inters}. We start by noticing that due to the structure of $X$, the equality $X = (P^1 \cup P^2) \cap \Z^3$ also holds, and from this, we get $\cconv(X) = \cconv\bigl( \cconv(P^1 \cap \Z^3) \cup \cconv(P^2 \cap \Z^3) \bigr)$. It can be verified that (for instance, see the sets $K^1$ and $K^2$ in Example~1 in~\cite{dey2013some})
    \begin{itemize}
        \item $\cconv(P^1 \cap \Z^3) = \{(x,y,z) \colon \sqrt{2}x - y \geq 0, \, y \geq 0, \, z= 0 \}$ and thus $\cconv(P^1 \cap \Z^3) \cap \{(x,y,z) \colon - \sqrt{2}x + y \geq 0\} = \{(x,y,z) \colon \sqrt{2}x - y = 0, \, z = 0\}$. 
        \item $\cconv(P^2 \cap \Z^3) = \conv(P^2 \cap \Z^3)$ and thus $\cconv(P^2 \cap \Z^3) \cap \{(x,y,z) \colon -\sqrt{2}x + y \geq 0\} = \varnothing$. Moreover, the recession cone of $\cconv(P^2 \cap \Z^3)$ is equal to $\cconv(P^1 \cap \Z^3)$.  
    \end{itemize}
    This implies immediately that $\cconv(X) \cap \{(x,y,z) \colon -\sqrt{2}x + y \geq 0\} \supseteq \{(x,y,z) \colon \sqrt{2}x - y = 0, \, z = 0\}$. To prove the inclusion in the other direction, we note that
    \[
        \cconv(P^1 \cap \Z^3) \cup \cconv(P^2 \cap \Z^3) \subseteq \{(x,y,z) \colon \sqrt{2}x - y \geq 0, \, z\geq 0\} \, ,
    \]
    because every point of $P^1$ and every point of $P^2$ satisfy $\sqrt{2}x \geq y$. From this, we deduce that $\cconv(X) \cap \{(x,y,z) \colon -\sqrt{2}x + y \geq 0\} \subseteq \{(x,y,z) \colon \sqrt{2}x - y = 0, \, z \geq 0\}$. We complete the argument as follows: since the recession cones of $\cconv(P^1 \cap \Z^3)$ and $\cconv(P^2 \cap \Z^3)$ are the same, which implies \cite[Corollary~9.8.1]{rockafellar:1970} that $\cconv(X) = \conv\bigl( \cconv(P^1 \cap \Z^3) \cup \conv(P^2 \cap \Z^3) \bigr)$ (we can get rid of the closure on the right-hand-side); if 
    {$\cconv(X) \cap \{(x,y,z) \colon \sqrt{2}x - y = 0, \, z > 0\}$} were non-empty, then there would be a point 
    $(x',y',z')$ in $\cconv(P^2 \cap \Z^3)$ satisfying 
    $\sqrt{2}x' - y' = 0$, which is not possible. \exampleqed
\end{example}

\subsection{Sufficient conditions for Geoffrion's theorem to hold.}\label{subsec:contributions}

{
The examples of Section~\ref{subsec:ex} are ``degenerate'' in the sense that the set of feasible solutions of~\eqref{eq:geoffrion-closed} is located on the boundary of $\conv(X)$. This is not coincidental. The relative interior of $\conv(X)$ having a non-empty intersection with $\{x\in \R^n \colon Ax \geq b\}$ is nothing else than Slater's condition, and basic results from convex programming ensure that the following holds. 
For the sake of completeness, we state this result formally below (for a proof, see Section~\ref{sec:prelim_results}).

\begin{proposition}\label{prop:slater}
    Suppose that the relative interior of $\conv(X)$ has a non-empty intersection with $\{x\in \R^n \colon Ax \geq b\}$. Then $v^L = \bar v^\star = v^\star$.
\end{proposition}}

The challenge for extending Geoffrion's theorem beyond finiteness and rationality is thus to understand integer programs whose feasibility set lies on the boundary of $\conv(X)$.

Our three main results in this regard are presented next.

When $X$ is the set of integer points of a non-rational polyhedron, $\cconv(X)$ is not necessarily a polyhedron. The first result deals with the case when $\cconv(X)$ is {\em locally polyhedral}, namely, that its intersection with every box is a polytope.

\begin{theorem}\label{thm:loc-poly}
    Suppose that $\cconv(X)$ is locally polyhedral. If~\eqref{eq:geoffrion-closed} admits an optimal solution, then $v^L = \bar v^\star$.
\end{theorem}
Note that while our original aim was to explore the case where $X$ is the set of integer points in an unbounded polyhedron which is not necessarily rational, the above result provides sufficient conditions even is cases where $X$ is the set of integer points in a non-polyhedral set. As an example, consider the set $X \coloneqq \{ (x,y) \in \Z^2\colon y \geq x^2\}$. Then it is not difficult to verify that $\cconv(X)$ is locally polyhedral~\cite{dey2013some} and $X \neq P \cap \mathbb{Z}^2$ for any polyhedron $P$. To see this notice that since $\max\{x \colon (x,y) \in X\}$ is unbounded, the recession cone of any polyhedron $P$ containing $X$ must contain a vector of the form $(v_x, v_y)$ where both $v_x > 0$ and $v_y > 0$. Now it is not difficult to show that such a polyhedron will contain integer points that are not in $X$, for example by using {Lemma~2.2} from~\cite{basu2010maximal}.

In fact, our proof of Theorem~\ref{thm:loc-poly} is quite general and does not even assume that $X$ is a set of integer points. 

The second result states that when the 
feasible region is included in the plane, then Geoffrion's theorem holds.

\begin{theorem}\label{thm:dim2}
    Suppose that $n \leq 2$. If~\eqref{eq:master} is feasible, then $v^L = \bar v^\star$.
\end{theorem}

Example~\ref{ex:ex1} shows that under the conditions of Theorems~\ref{thm:loc-poly} and~\ref{thm:dim2}, while $v^L = \bar v^\star$ must hold, the equality $v^L = v^\star$ may not necessarily hold.

The third result somehow shows a reverse phenomenon 
with regards to the standard assumptions for Geoffrion's theorem to hold: instead of requiring the rationality of the constraints defining $X$, we require rationality for the other constraints.





\begin{theorem}\label{thm:rationa-constr}
    Suppose that the system $Ax \geq b$ is a rational single-row system. If $\inf\{c \cdot x \colon x \in X\}$ is finite, then $v^L = \bar v^\star = v^\star$.
\end{theorem}

\section{Proofs}

\subsection{Preliminary results}\label{sec:prelim_results}

Even when the conclusion of Geoffrion's theorem does not hold, the best bound given by Lagrangian relaxation and the optimal values of~\eqref{eq:geoffrion} and~\eqref{eq:geoffrion-closed} are always ordered is some way.

\begin{proposition}\label{prop:ineq}
    The following inequalities always hold: $v^L \leq \bar v^\star \leq v^\star$.
\end{proposition}

\begin{proof}
    For every feasible solution $x$ of~\eqref{eq:geoffrion-closed} and every $\lambda \in \R_+^m$, we have \[
    \G(\lambda) \leq c \cdot x + \lambda \cdot (Ax - b) \leq c \cdot x \, .
    \]
    This shows that the inequality $v^L \leq \bar v ^\star$ holds. The inequality $\bar v^\star \leq v^\star$ is immediate.
\end{proof}

We finish this subsection with a proof of Proposition~\ref{prop:slater}. 

\begin{proof}[Proof of Proposition~\ref{prop:slater}]
    If $v^\star = -\infty$, then Proposition~\ref{prop:ineq} ensures that the desired equalities hold. Suppose thus that $v^\star > -\infty$. The intersection of the relative interior of $\conv(X)$ with $\{x\in \R^n \colon Ax \geq b\}$ being non-empty is precisely the (relaxed) Slater's condition, which implies~\cite{BenTalNemirovski2023} together with $v^\star > -\infty$ that strong duality holds for~\eqref{eq:geoffrion}. The function $\G$ being the dual function of~\eqref{eq:geoffrion} as well (the infimum of a linear function on $X$ does not change when taken on $\conv(X)$), we get $v^L = v^\star$, and we conclude with Proposition~\ref{prop:ineq} again.
\end{proof}

\subsection{Proof of Theorem~\ref{thm:loc-poly}}


\begin{proof}[Proof of Theorem~\ref{thm:loc-poly}]
Suppose that \eqref{eq:geoffrion-closed} admits an optimal solution $x^\star$. Let $Q$ be a full-dimensional box with $x^\star$ in its interior. The point $x^\star$  is a global minimum of $c\cdot x$ on the polyhedron $K=\{ x\in\R^n\colon Ax \geq b\}\cap (Q\cap\cconv(X))$. Denote by $D x \geq e$ the active inequalities of the polyhedron $Q\cap\cconv(X)$ at $x^\star$. Since $c \cdot x\geq \bar v^\star$ is a valid inequality for $K$ and contains the point $x^\star\in K$, there exist $\lambda,\mu$ with nonnegative entries such that $\lambda^{\top}A+\mu^{\top}D=c^{\top}$ and $\lambda \cdot b+ \mu \cdot e\geq \bar v^\star$ (e.g., with Farkas's lemma).

As $x^\star$ is in the interior of $Q$, the active inequalities $D x \geq e$ do not contain any inequality defining $Q$ and thus $\cconv(X)\subseteq\{x\colon D x \geq e\}$. This is used for the inequality before the last one in the following chain of equalities and inequalities:
\[
v^L \geq \G(\lambda) {\geq} \lambda \cdot b+\inf_{x \in \cconv(X)}(c^{\top}-\lambda^{\top} A) x = \lambda\cdot b+\inf_{x\in\cconv(X)}\mu^{\top}Dx \geq \lambda\cdot b+\mu \cdot e \geq \bar v^\star \, .
 \]
 (The second inequality is actually an equality: the infimum of a linear function on $X$ does not change when taken on $\conv(X)$, and by continuity does not change either  when taken on $\cconv(X)$. But this is not needed for the proof to hold.)
 The reverse inequality $v^L\leq \bar v^\star$ results from Proposition~\ref{prop:ineq}.
\end{proof}

\subsection{Proof of Theorem~\ref{thm:dim2}}

The proof relies on a preliminary lemma.

\begin{lemma}\label{lem:local2} The set $\cconv(X)$ is locally polyhedral for every subset $X \subseteq \Z^2$.
\end{lemma}

\begin{proof}
Let $K=\cconv(X)$. To show that $K$ is locally polyhedral, we must show that $K\cap P$ is a polytope for every box $P\subseteq \R^2$. The extreme points of $K\cap P$ are either: \begin{enumerate*}[label=(\roman*)]
    \item\label{extrKP} extreme points of $K$ contained in $P$, \item\label{extrP} extreme points of $P$, or \item \label{edges} intersections of edges of $P$ with the boundary of $K$. 
 \end{enumerate*} Since $X$ is a closed, set the extreme points of $K$ belong to $X\subseteq\Z^n$ (see, for instance, Theorem 3.5 in~\cite{Klee1957}). This implies that the extreme points of $K$ contained in $P$ are finitely many. Therefore, $K\cap P$ has finitely many extreme points of type~\ref{extrKP} or type~\ref{extrP}. Finally, the intersection of each edge of $P$ with the boundary of {$K\subseteq\R^2$} can contain at most two extreme points of $K\cap P$. Therefore, since $P$ has finitely many edges, $K\cap P$  has finitely many extreme points of type~\ref{edges}.  We conclude that $K\cap P$ has finitely many extreme points, and thus it is a polytope.
\end{proof}

Note that the statement of the above result does not hold in $\mathbb{R}^3$; see example in~\cite{dey2013some}.

\begin{proof}[Proof of Theorem~\ref{thm:dim2}]
Let $Y \coloneqq \cconv(X) \cap \{x \in \R^n \colon A x \geq  b \}$. Only three cases are possible. If \eqref{eq:geoffrion-closed} admits an optimal solution, then we get the conclusion by combining Lemma~\ref{lem:local2} and Theorem~\ref{thm:loc-poly}. If $Y$ is of dimension $1$ and \eqref{eq:geoffrion-closed} does not admit an optimal solution, then $Y$ is a half-line or a line, we have $\bar v^\star = -\infty$, and we get the conclusion by Proposition~\ref{prop:ineq}. If $Y$ is of dimension $2$, then $Y$ has a non-empty interior, {which shows that the relative interior of $\conv(X)$ has a non-empty intersection with $\{ x\in\R^n\colon A x\geq b\}$, and then we conclude with Proposition~\ref{prop:slater}.}
\end{proof}

\subsection{Proof of Theorem~\ref{thm:rationa-constr}}

\begin{proof}[Proof of Theorem~\ref{thm:rationa-constr}]
Without loss of generality, we assume that all data describing $Ax\geq b$ is integral. Denote by $a^\top$ the single row of $A$ and, to ease notation, for $\beta \in \Z$, set 
\[
H^\beta \coloneqq \{x \in \R^n \colon a \cdot x = \beta\}\, ,\quad H^{\leq \beta} \coloneqq \{x \in \R^n \colon a \cdot x \leq \beta\}\, , \quad H^{\geq \beta} \coloneqq \{x \in \R^n \colon a \cdot x \geq \beta\}\, .
\] There are three possible cases:
\begin{enumerate}[label=(\roman*)]
    \item\label{one} $X \subseteq H^{\geq b}$.
    \item\label{two} $H^{\beta_1} \cap X \neq \varnothing$ and $H^{\beta_2} \cap X\neq \varnothing$ for $b \leq \beta_1 < \beta_2$.
    \item\label{one-border} $H^{\geq b} \cap X\subseteq H^b$.
\end{enumerate}

{
In case~\ref{one}, we have $\conv(X) \subseteq H^{\geq b}$, and since $X$ is non-empty ($\inf\{c\cdot x\colon x \in X\}$ being finite), we conclude with Proposition~\ref{prop:slater}. In case~\ref{two}, there are points in $\{x \in \R^n \colon \beta_1 < a \cdot x < \beta_2\}$ that belong to $\conv(X)$. By Theorem~6.1 in~\cite{rockafellar:1970}, every point in $\conv(X)$ is arbitrarily close to the relative interior of $\conv(X)$, and we conclude again with Proposition~\ref{prop:slater}.} In case~\ref{one-border}, we have \[
v^L \geq \min \left( \inf\{c \cdot x \colon x \in H^b \cap X\},
\inf\{c \cdot x + \lambda(b - a \cdot x) \colon x \in H^{\leq b-1} \cap X\}\right) \, .
\]
The first term in the minimum is lower-bounded by { $v^\star$. In the second term, the expression $b-a\cdot x$ is at least $1$. Thus, for every $\lambda \in \R_+$, we have
\[
v^L \geq \min \left( v^\star,
\lambda + \inf\{c \cdot x \colon x \in X\}\right) \, .
\]
The quantity $\inf\{c \cdot x \colon x \in X\}$ being finite, choosing $\lambda$ large enough shows $v^L \geq v^\star$. We conclude with Proposition~\ref{prop:ineq}.}
\end{proof}

\section{Concluding remarks}

In this paper, we investigate Geoffrion’s theorem ``beyond finiteness and rationality,'' that is, when the set $X$ of integer points is infinite and cannot be described as the integer points in a rational polyhedron. Although our results are quite general, we are left with some open questions regarding the validity of Geoffrion’s theorem in certain situations.

The first open question concerns Theorem~\ref{thm:loc-poly}. We do not know if this result is valid or not in the case the optimization problem~\eqref{eq:geoffrion-closed} does not admit an optimal solution. Finding a proof or a counterexample for this result in the case~\eqref{eq:geoffrion-closed} is bounded below but no optimal solution exists is an interesting question. We note here that~\eqref{eq:geoffrion-closed} having an optimal solution is a crucial assumption in our proof.

The second open question is about Theorem~\ref{thm:rationa-constr}. In this result, the system $Ax\geq b$ is a rational single-constraint system. We do not know if the result is valid for rational systems defined by more than one constraint. A possible proof approach for the more general case, is to find a rational hyperplane $a\cdot x=\beta$ separating $P\coloneqq\{x\colon Ax\geq b\}$ and $\cconv(X)$ and then using this hyperplane to repeat the argument in our proof of Theorem~\ref{thm:rationa-constr}. However, this separating hyperplane is not guaranteed to exist; in the following example, no rational hyperplane separating $P$ and $\cconv(X)$ exists, although Geoffrion's theorem holds true. 

\begin{example}
Consider the half-line $P \coloneqq \{(x,y,z)\in \R^3\colon x=1,y=1,z \geq 1\}$, the non-rational half-space $T \coloneqq \{(x,y,z)\in \R^3 \colon \sqrt{2}(x-y) + z \leq 1\}$, and let $X \coloneqq T \cap \Z^n$. We have that $\convc(X)=T$ since the set $\{(x,y,z)\in \R^3 \colon 1-\varepsilon\leq\sqrt{2}(x-y) + z \leq 1\}$ cannot be lattice-free for any $\varepsilon>0$ (see, for instance, Theorem 2 in~\cite{basu2010MinIneqRational}). Clearly, the only hyperplane separating $P$ and $\convc(X)$ is $\sqrt{2}(x-y) + z = 1$, which is a non-rational hyperplane. 

Consider the integer program
 \begin{equation}\label{eq:ex3}
   \begin{array}{rl}
        \text{minimize} & z \\
        \text{s.t.} &x=1 \\
         & y=1 \\
          & z \geq 1 \\
        & (x,y,z) \in X \,.
    \end{array}
    \end{equation}
Since $P\cap \convc(X)=P\cap \conv(X)=\{(1,1,1)\}$, we have that $\bar v^\star = v^\star=1$. Moreover, since $\convc(X)$ is a polyhedron such that problem~\eqref{eq:geoffrion-closed} admits an optimal solution ($x^*=y^*=z^*=1$), by Theorem~\ref{thm:loc-poly} we obtain $v^L= \bar v^\star = v^\star$. \exampleqed
\end{example}


Finally, we note that identifying additional sufficient conditions for Geoffrion's theorem to hold would be of interest. Since the polyhedrality of integer hulls of rational polyhedra is related to the finiteness of the Hilbert basis of rational polyhedral cones, a possible direction for further study is to understand whether there are connections between Geoffrion's theorem and the recently studied sufficient conditions for finitely generated sets of integer points under the action of a finitely generated group~\cite{de2025integer}.

\subsection*{Acknowledgments} Part of this work was done while the authors were benefiting from a Collaborate@ICERM support by the Institute for Computational and Experimental Research in Mathematics in Providence, RI, during August 2024. 
\bibliographystyle{plain}
\bibliography{Refs}

\end{document}